\documentclass[12,fleqn]{article}
\usepackage{hyperref}
\usepackage{amsthm}
\usepackage{amsfonts}
\usepackage{amsmath}
\usepackage{amssymb}
\usepackage{latexsym}
\usepackage{fancyhdr}
\usepackage{nicefrac}
\usepackage{multicol}

\usepackage{pgf}
\usepackage{tikz}
\usetikzlibrary{arrows,automata}
\usepackage[latin1]{inputenc}

\theoremstyle{definition}
\newtheorem{defi}{Definition}

\newtheorem{ex}[defi]{Example}

\theoremstyle{plain}
\newtheorem{thm}[defi]{Theorem}
\newtheorem{lem}[defi]{Lemma}

\newtheorem{alg}[defi]{Algorithm}

\title{Skolem Circles}

\author{James Bubear, Joanne L. Hall\\
\\
School of Mathematical Sciences\\
Queensland University of Technology
}

\begin{document}

\maketitle

\section*{Abstract}
Skolem sequences and Skolem labeled graphs have been described and examined for several decades.  This note explores weak Skolem labelling of cycle graphs, which we call Skolem circles.  The relationship between Skolem sequences and Skolem cirlces is explored, and Skolem circles of small sizes are enumerated, with some loose general bounds established.

\section{Introduction}

A Skolem-type sequence of order $M$ is a sequence $(s_1, s_2, . . . , s_m)$  of positive integers $i \in  D$ with $|D|=m$ such that for each $i\in D$ there is exactly one $ j \in 
\{1, 2, . . .,m-i\}$ such that $s_j = s_j+i = i$.

 If the set $D=\{1,2,3,\dots, m\}$ then $S$ is a {\em Skolem sequence} or order $m$, if the set $D=\{d,d+1,\dots, d+m-1\}$ then $S$ is a {\em Langford sequence} of order $m$ and  defect $d$.   A sequence may also contain a null element often denoted as $0$, if the null element appears in the penultimate position of the sequence it is called a  $hook$.   A Skolem-type sequence is $k$ {\em extended}  if it contains exactly one null symbol which is in position $k$. 
We consider $(0)$ to be the extended Skolem-type sequence of order $0$.

Skolem and Langford sequences were first introduced by Langford in \cite{Langford1958} and  Skolem \cite{skolem1957} in the 1950s.  Skolem and Langford sequences have been used to construct Steiner Triple Systems \cite{skolem1957}, difference sets \cite{abrham1981}
and have many applications in graph theory \cite{mendelsohn1971}\cite{harary1988}\cite{francetic2017}.

There are many sequences with the Skolem property, each interesting on its own or for particular applications \cite{shalaby2007}\cite{francetic2009}.   Another variant of a Skolem sequences is a Skolem labeling of a graph.

A {\em Skolem Labelled Graph} is a graph  $G$ with $2m$ nodes, each  node having a label from the set $\{1,2,\dots, m\}$ such that 
\begin{enumerate}
\item   each label appears exactly twice
\item if any two nodes $v_1,v_2$ have the same label $s$, then $d(v_1,v_2)=s$
\item removing any edge  from the graph violates condition $2$.
\end{enumerate}

In this paper we are interested in Skolem labelling of cycle graphs.  However we use a {\em weak Skolem labeling} including labellings that would violate condition 3.

A Skolem sequence of order $m$ may be represented as a sequence of $2m$ symbols such as \\ $(1 ,1, 4, 2, 3 ,2 ,4 ,3)$, or a set of $m$ ordered pairs $\{(a_1,b_1),(a_2,b_3),\dots,(a_m,b_m)\}$ where $(a_i, b_i)$ are the locations of the symbol $i$ in the sequence with $a_i<b_i$ such as $\{(1,2),(4,6),(5,8),(3,7)\}$.  Thus $b_i-a_i=i$.  
We consider a variation where the arrangement is a circle rather than a sequence.  A {\em Skolem circle} may be represented as a labeling of a circle graph on $2m$ nodes as  in Figure \ref{fig:6sequences}, or as a set of ordered pairs $\{(a_1,b_1),(a_2,b_2),\dots,(a_m,b_m)\}$ where $b_i-a_i\equiv i\mod 2m$. Note that $a_i>b_1$ if and only if $b_i\leq i$.   A Skolem circle may be suitable to use in applications where a finite cyclic group is useful. 

In Section \ref{sec:circle} we list some known results on Skolem sequences and the analogous results for Skolem circles.  In Section \ref{sec:structure} we investigate the structure of Skolem Circles showing how they can be broken into subsequences of Skolem-type. Computational techniques are used in  Section \ref{sec:enumerate} to enumerate distinct Skolem circles of small orders.

\begin{figure}[h]
\begin{center}
\begin{tikzpicture}
\def \n {8}
\def \radius {2cm}
\def \margin {8} 

\def \s{0} 
\def \e{1}
  \node[draw, circle] at ({360/\n * (\s)}:\radius) {\e}; 
 \draw[-, >=latex]   ({360/\n * (\s-1)+\margin}:\radius) 
    arc ({360/\n * (\s-1)+\margin}:{360/\n * (\s)-\margin}:\radius);
\def \s{1} 
\def \e{1}
  \node[draw, circle] at ({360/\n * (\s)}:\radius) {\e}; 
 \draw[-, >=latex]   ({360/\n * (\s-1)+\margin}:\radius) 
    arc ({360/\n * (\s-1)+\margin}:{360/\n * (\s)-\margin}:\radius);
\def \s{2} 
\def \e{4}
  \node[draw, circle] at ({360/\n * (\s)}:\radius) {\e}; 
 \draw[-, >=latex]   ({360/\n * (\s-1)+\margin}:\radius) 
    arc ({360/\n * (\s-1)+\margin}:{360/\n * (\s)-\margin}:\radius);
\def \s{3} 
\def \e{2}
  \node[draw, circle] at ({360/\n * (\s)}:\radius) {\e}; 
 \draw[-, >=latex]   ({360/\n * (\s-1)+\margin}:\radius) 
    arc ({360/\n * (\s-1)+\margin}:{360/\n * (\s)-\margin}:\radius);
\def \s{4} 
\def \e{3}
  \node[draw, circle] at ({360/\n * (\s)}:\radius) {\e}; 
 \draw[-, >=latex]   ({360/\n * (\s-1)+\margin}:\radius) 
    arc ({360/\n * (\s-1)+\margin}:{360/\n * (\s)-\margin}:\radius);
\def \s{5} 
\def \e{2}
  \node[draw, circle] at ({360/\n * (\s)}:\radius) {\e}; 
 \draw[-, >=latex]   ({360/\n * (\s-1)+\margin}:\radius) 
    arc ({360/\n * (\s-1)+\margin}:{360/\n * (\s)-\margin}:\radius);
\def \s{6} 
\def \e{4}
  \node[draw, circle] at ({360/\n * (\s)}:\radius) {\e}; 
  \draw[-, >=latex]   ({360/\n * (\s-1)+\margin}:\radius) 
    arc ({360/\n * (\s-1)+\margin}:{360/\n * (\s)-\margin}:\radius);
\def \s{7} 
\def \e{3}
  \node[draw, circle] at ({360/\n * (\s)}:\radius) {\e}; 
  \draw[-, >=latex]   ({360/\n * (\s-1)+\margin}:\radius) 
    arc ({360/\n * (\s-1)+\margin}:{360/\n * (\s)-\margin}:\radius);

\end{tikzpicture}
\end{center}
\caption{A weak Skolem labeling of a cycle graph on $8$ nodes, also called a Skolem circle on 4 symbols. There are three different edges that could be removed to form a Skolem labeling, hence  this Skolem circle contains 6 distinct Skolem sequences. \label{fig:6sequences}}
\end{figure}
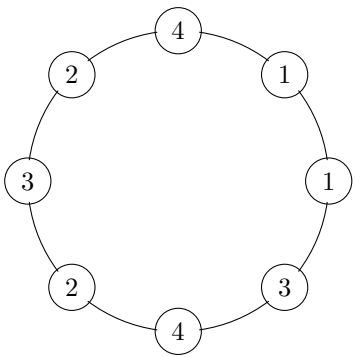
\section{Skolem Circles \label{sec:circle}}

A {\em Skolem circle} is a  weakly Skolem labeled  cycle graph.  A Skolem circle is a cycle graph,   $G$, with $2m$ nodes, each  node having a label from the set $\{1,2,\dots, m\}$ such that 
\begin{enumerate}
\item   each label appears exactly twice
\item if any two nodes $v_1,v_2$ have the same label $s$, then $d(v_1,v_2)=s$
\end{enumerate}

 A  Skolem circle may be represented as a labeling of a cycle graph on $2m$ nodes as  in Figure \ref{fig:6sequences}, or as a set of ordered pairs $\{(a_1,b_1),(a_2,b_2),\dots,(a_m,b_m)\}$ where 
\[b_i-a_i\equiv i\mod 2m.\] This inspiration for this definition comes from using modular arithmetic, rather than integer arithmetic.

A Skolem sequence may be wrapped around a cycle to form a Skolem circle.  For example the sequence (1, 1, 4, 2, 3 ,2 ,4 ,3) can be wrapped around to form the circle in Figure \ref{fig:6sequences}.

Skolem sequences have been studied for a few decades now, there is quite a lot known about them.  Some of the properties of Skolem sequences have analogous properties of Skolem circles.
\begin{thm}\cite{skolem1957}\label{thm:skolem}
A Skolem Sequence of order $m$ exists if and only if $m\equiv 0, 1\mod 4$.
\end{thm}

In \cite{mendelsohn1991} it is shown that  Skolem labeled cycle graphs exist for every cycle on $2m$ nodes if and only if 
 $m\geq 8$
and   $m\equiv 0, 1\mod 4$.  Those Skolem circles which are weak Skolem lableld graphs (and not Skolem labeld graphs) can be constructed by wrapping a Skolem sequence around the circle graph, thus the possible orders of Skolem circles are the same as the possible orders of Skolem sequences.

\begin{thm}
A Skolem circle of order $m$ exists if and only if $m\equiv 0, 1\mod 4$.
\end{thm}

Some Skolem circles contain several Skolem sequences.  The Skolem circle in Figure \ref{fig:6sequences} contains six distinct Skolem sequences.
There are three Anti-clockwise sequences $(1, 1, 4, 2, 3, 2, 4, 3)$, $(4, 2, 3, 2, 4, 3, 1, 1)$, $(2, 3, 2, 4, 3, 1, 1, 4)$,
and three clockwise sequences $(3, 4, 2, 3, 2, 4, 1, 1)$, $(1, 1, 3, 4, 2, 3, 2, 4)$,\\
 $(4, 1, 1, 3, 4, 2, 3, 2)$.

A circle is a highly symmetric arrangement of cells.  Fixing a positional labelling on a circle can be useful to discuss aspects of the circle.  A positional labelling of the circle  contains the symbols $\{1,2,\dots, 2m\}$ in lexicographic order, beginning with any vertex, and moving either clockwise or anticlockwise. 
Allowing travel around a circle clockwise or anti clockwise and allowing beginning vertex we consider Skolem sequences to be { \em circle equivalent } if they form  the same Skolem circle  allowing for cyclic shifts and reversal of direction. Skolem circles $C_1=\{(a_1,b_1),(a_2,b_2),\dots,(a_m,b_m)\}$ and $C_2=\{(\alpha_1,\beta_1),(\alpha_2,\beta_2),\dots,(\alpha_m,\beta_m)\}$ are circle equivalent if there exists $x$ such that $(a_i+x,b_i+x)=(\alpha_i,\beta_i)\mod 2m$ for all $1\leq i\leq m$ or $(x-a_i,x-b_i)=(\alpha_i,\beta_i)\mod 2m$ for all $1\leq i\leq m$.    

 For ease of comparing Skolem circles we consider a Skolem circle to be given the { \em standard positional labelling} if $(a_1,b_1)=(1,2)$ and $3\leq a_2\leq m$.  
When drawing graphs will represent the labeling with vertex $1$ at an angle of $0$ from the horizontal, and  vertex at an angle of $x\frac{\pi}{m}$ being given label $x+1$.  The Skolem circles in Figures \ref{fig:6sequences} and \ref{fig:unbreakable} are both shown in the standard positional labeling.  The standard positional labeling also allows representation of a Skolem circle as a sequence.  The Skolem circle of Figure \ref{fig:6sequences} may be represented as the sequence $(1, 1, 4, 2, 3, 2, 4, 3)$.  The six Skolem sequences $(1, 1, 4, 2, 3, 2, 4, 3)$, $(4, 2, 3, 2, 4, 3, 1, 1)$, $(2, 3, 2, 4, 3, 1, 1, 4)$, $(3, 4, 2, 3, 2, 4, 1, 1)$,\\ $(1, 1, 3, 4, 2, 3, 2, 4)$, $(4, 1, 1, 3, 4, 2, 3, 2)$ are circle equivalent.  This notion of circle equivalence will be used in Section \ref{sec:enumerate} to put a lower bound on the number of distinct Skolem circles.

There exists Skolem circles which are Skolem labelling of a cycle graph, and hence cannot be constructed by wrapping a Skolem sequence around a cycle. Figure \ref{fig:unbreakable} shows a Skolem labelling of  a cycle graph with $16$ vertices.
\begin{figure}[h]
\begin{center}
\begin{tikzpicture}
\def \n {16}
\def \radius {3cm}
\def \margin {8} 

\def \s{0} 
\def \e{1}
  \node[draw, circle] at ({360/\n * (\s)}:\radius) {\e}; 
   \draw[-, >=latex]   ({360/\n * (\s-1.1)+\margin}:\radius) 
    arc ({360/\n * (\s-1.1)+\margin}:{360/\n * (\s+.1)-\margin}:\radius);
\def \s{1} 
\def \e{1}
  \node[draw, circle] at ({360/\n * (\s)}:\radius) {\e}; 
   \draw[-, >=latex]   ({360/\n * (\s-1.1)+\margin}:\radius) 
    arc ({360/\n * (\s-1.1)+\margin}:{360/\n * (\s+.1)-\margin}:\radius);
\def \s{2} 
\def \e{4}
  \node[draw, circle] at ({360/\n * (\s)}:\radius) {\e}; 
   \draw[-, >=latex]   ({360/\n * (\s-1.1)+\margin}:\radius) 
    arc ({360/\n * (\s-1.1)+\margin}:{360/\n * (\s+.1)-\margin}:\radius);
\def \s{3} 
\def \e{8}
  \node[draw, circle] at ({360/\n * (\s)}:\radius) {\e}; 
   \draw[-, >=latex]   ({360/\n * (\s-1.1)+\margin}:\radius) 
    arc ({360/\n * (\s-1.1)+\margin}:{360/\n * (\s+.1)-\margin}:\radius);
\def \s{4} 
\def \e{7}
  \node[draw, circle] at ({360/\n * (\s)}:\radius) {\e}; 
   \draw[-, >=latex]   ({360/\n * (\s-1.1)+\margin}:\radius) 
    arc ({360/\n * (\s-1.1)+\margin}:{360/\n * (\s+.1)-\margin}:\radius);
\def \s{5} 
\def \e{5}
  \node[draw, circle] at ({360/\n * (\s)}:\radius) {\e}; 
   \draw[-, >=latex]   ({360/\n * (\s-1.1)+\margin}:\radius) 
    arc ({360/\n * (\s-1.1)+\margin}:{360/\n * (\s+.1)-\margin}:\radius);
\def \s{6} 
\def \e{4}
  \node[draw, circle] at ({360/\n * (\s)}:\radius) {\e}; 
    \draw[-, >=latex]   ({360/\n * (\s-1.1)+\margin}:\radius) 
    arc ({360/\n * (\s-1.1)+\margin}:{360/\n * (\s+.1)-\margin}:\radius);
\def \s{7} 
\def \e{2}
  \node[draw, circle] at ({360/\n * (\s)}:\radius) {\e}; 
   \draw[-, >=latex]   ({360/\n * (\s-1.1)+\margin}:\radius) 
    arc ({360/\n * (\s-1.1)+\margin}:{360/\n * (\s+.1)-\margin}:\radius);
\def \s{8} 
\def \e{6}
  \node[draw, circle] at ({360/\n * (\s)}:\radius) {\e}; 
   \draw[-, >=latex]   ({360/\n * (\s-1.1)+\margin}:\radius) 
    arc ({360/\n * (\s-1.1)+\margin}:{360/\n * (\s+.1)-\margin}:\radius);
\def \s{9} 
\def \e{2}
  \node[draw, circle] at ({360/\n * (\s)}:\radius) {\e}; 
   \draw[-, >=latex]   ({360/\n * (\s-1.1)+\margin}:\radius) 
    arc ({360/\n * (\s-1.1)+\margin}:{360/\n * (\s+.1)-\margin}:\radius);
\def \s{10} 
\def \e{5}
  \node[draw, circle] at ({360/\n * (\s)}:\radius) {\e}; 
   \draw[-, >=latex]   ({360/\n * (\s-1.1)+\margin}:\radius) 
    arc ({360/\n * (\s-1.1)+\margin}:{360/\n * (\s+.1)-\margin}:\radius);
\def \s{11} 
\def \e{8}
  \node[draw, circle] at ({360/\n * (\s)}:\radius) {\e}; 
   \draw[-, >=latex]   ({360/\n * (\s-1.1)+\margin}:\radius) 
    arc ({360/\n * (\s-1.1)+\margin}:{360/\n * (\s+.1)-\margin}:\radius);
\def \s{12} 
\def \e{3}
  \node[draw, circle] at ({360/\n * (\s)}:\radius) {\e}; 
   \draw[-, >=latex]   ({360/\n * (\s-1.1)+\margin}:\radius) 
    arc ({360/\n * (\s-1.1)+\margin}:{360/\n * (\s+.1)-\margin}:\radius);
\def \s{13} 
\def \e{7}
  \node[draw, circle] at ({360/\n * (\s)}:\radius) {\e}; 
   \draw[-, >=latex]   ({360/\n * (\s-1.1)+\margin}:\radius) 
    arc ({360/\n * (\s-1.1)+\margin}:{360/\n * (\s+.1)-\margin}:\radius);
\def \s{14} 
\def \e{6}
  \node[draw, circle] at ({360/\n * (\s)}:\radius) {\e}; 
   \draw[-, >=latex]   ({360/\n * (\s-1.1)+\margin}:\radius) 
    arc ({360/\n * (\s-1.1)+\margin}:{360/\n * (\s+.1)-\margin}:\radius);
\def \s{15} 
\def \e{3}
  \node[draw, circle] at ({360/\n * (\s)}:\radius) {\e}; 
   \draw[-, >=latex]   ({360/\n * (\s-1.1)+\margin}:\radius) 
    arc ({360/\n * (\s-1.1)+\margin}:{360/\n * (\s+.1)-\margin}:\radius);
\end{tikzpicture}
\end{center}
\caption{A Skolem labelling of a cycle graph on 16 vertices. \label{fig:unbreakable}}
\end{figure}
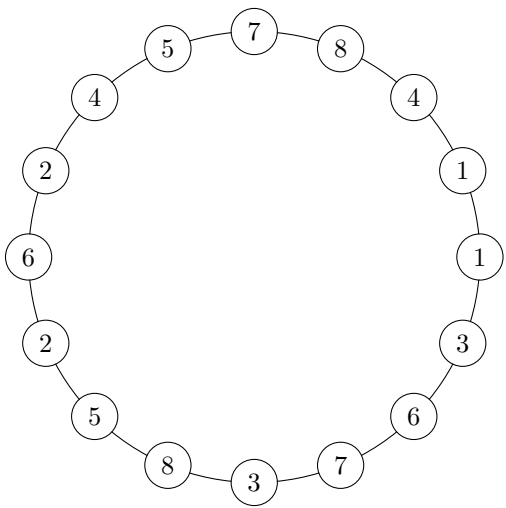

\section{Structures of Skolem Circles \label{sec:structure}}
Skolem sequences can contain subsequences which are Skolem-type sequences or extended Skolem type sequences.  This is apparent when regarding a Skolem circle as a weak Skolem labeled cycle graph.  Removing appropriate edges from a weak Skolem labeled graph creates a Skolem labeled graph, although this graph may no longer be a cycle, and if more than one edge is removed the Skolem labeled graph is not connected (see Figure \ref{fig:disconnected}).

\begin{figure}[t]

\begin{multicols}{3}

\begin{tikzpicture}
\def \n {8}
\def \radius {2cm}
\def \margin {8} 
\def \s{0} 
\def \e{1}
  \node[draw, circle] at ({360/\n * (\s)}:\radius) {\e}; 
 \draw[-, >=latex]   ({360/\n * (\s-1)+\margin}:\radius) 
    arc ({360/\n * (\s-1)+\margin}:{360/\n * (\s)-\margin}:\radius);
\def \s{1} 
\def \e{1}
  \node[draw, circle] at ({360/\n * (\s)}:\radius) {\e}; 
 \draw[ , >=latex]   ({360/\n * (\s-1)+\margin}:\radius) 
    arc ({360/\n * (\s-1)+\margin}:{360/\n * (\s)-\margin}:\radius);
\def \s{2} 
\def \e{4}
  \node[draw, circle] at ({360/\n * (\s)}:\radius) {\e}; 
    arc ({360/\n * (\s-1)+\margin}:{360/\n * (\s)-\margin}:\radius);
\def \s{3} 
\def \e{2}
  \node[draw, circle] at ({360/\n * (\s)}:\radius) {\e}; 
 \draw[-, >=latex]   ({360/\n * (\s-1)+\margin}:\radius) 
    arc ({360/\n * (\s-1)+\margin}:{360/\n * (\s)-\margin}:\radius);
\def \s{4} 
\def \e{3}
  \node[draw, circle] at ({360/\n * (\s)}:\radius) {\e}; 
 \draw[-, >=latex]   ({360/\n * (\s-1)+\margin}:\radius) 
    arc ({360/\n * (\s-1)+\margin}:{360/\n * (\s)-\margin}:\radius);
\def \s{5} 
\def \e{2}
  \node[draw, circle] at ({360/\n * (\s)}:\radius) {\e}; 
 \draw[-, >=latex]   ({360/\n * (\s-1)+\margin}:\radius) 
    arc ({360/\n * (\s-1)+\margin}:{360/\n * (\s)-\margin}:\radius);
\def \s{6} 
\def \e{4}
  \node[draw, circle] at ({360/\n * (\s)}:\radius) {\e}; 
  \draw[-, >=latex]   ({360/\n * (\s-1)+\margin}:\radius) 
    arc ({360/\n * (\s-1)+\margin}:{360/\n * (\s)-\margin}:\radius);
\def \s{7} 
\def \e{3}
  \node[draw, circle] at ({360/\n * (\s)}:\radius) {\e}; 
  \draw[-, >=latex]   ({360/\n * (\s-1)+\margin}:\radius) 
    arc ({360/\n * (\s-1)+\margin}:{360/\n * (\s)-\margin}:\radius);
\end{tikzpicture}

\begin{tikzpicture}
\def \n {8}
\def \radius {2cm}
\def \margin {8} 
\def \s{0} 
\def \e{1}
  \node[draw, circle] at ({360/\n * (\s)}:\radius) {\e}; 
 \draw[-, >=latex]   ({360/\n * (\s-1)+\margin}:\radius) 
    arc ({360/\n * (\s-1)+\margin}:{360/\n * (\s)-\margin}:\radius);
\def \s{1} 
\def \e{1}
  \node[draw, circle] at ({360/\n * (\s)}:\radius) {\e}; 
 \draw[ , >=latex]   ({360/\n * (\s-1)+\margin}:\radius) 
    arc ({360/\n * (\s-1)+\margin}:{360/\n * (\s)-\margin}:\radius);
\def \s{2} 
\def \e{4}
  \node[draw, circle] at ({360/\n * (\s)}:\radius) {\e}; 
 \draw[-, >=latex]   ({360/\n * (\s-1)+\margin}:\radius) 
    arc ({360/\n * (\s-1)+\margin}:{360/\n * (\s)-\margin}:\radius);
\def \s{3} 
\def \e{2}
  \node[draw, circle] at ({360/\n * (\s)}:\radius) {\e}; 
    arc ({360/\n * (\s-1)+\margin}:{360/\n * (\s)-\margin}:\radius);
\def \s{4} 
\def \e{3}
  \node[draw, circle] at ({360/\n * (\s)}:\radius) {\e}; 
 \draw[-, >=latex]   ({360/\n * (\s-1)+\margin}:\radius) 
    arc ({360/\n * (\s-1)+\margin}:{360/\n * (\s)-\margin}:\radius);
\def \s{5} 
\def \e{2}
  \node[draw, circle] at ({360/\n * (\s)}:\radius) {\e}; 
 \draw[-, >=latex]   ({360/\n * (\s-1)+\margin}:\radius) 
    arc ({360/\n * (\s-1)+\margin}:{360/\n * (\s)-\margin}:\radius);
\def \s{6} 
\def \e{4}
  \node[draw, circle] at ({360/\n * (\s)}:\radius) {\e}; 
  \draw[-, >=latex]   ({360/\n * (\s-1)+\margin}:\radius) 
    arc ({360/\n * (\s-1)+\margin}:{360/\n * (\s)-\margin}:\radius);
\def \s{7} 
\def \e{3}
  \node[draw, circle] at ({360/\n * (\s)}:\radius) {\e}; 
  \draw[-, >=latex]   ({360/\n * (\s-1)+\margin}:\radius) 
    arc ({360/\n * (\s-1)+\margin}:{360/\n * (\s)-\margin}:\radius);
\end{tikzpicture}

\begin{tikzpicture}
\def \n {8}
\def \radius {2cm}
\def \margin {8} 
\def \s{0} 
\def \e{1}
  \node[draw, circle] at ({360/\n * (\s)}:\radius) {\e}; 
    arc ({360/\n * (\s-1)+\margin}:{360/\n * (\s)-\margin}:\radius);
\def \s{1} 
\def \e{1}
  \node[draw, circle] at ({360/\n * (\s)}:\radius) {\e}; 
 \draw[ , >=latex]   ({360/\n * (\s-1)+\margin}:\radius) 
    arc ({360/\n * (\s-1)+\margin}:{360/\n * (\s)-\margin}:\radius);
\def \s{2} 
\def \e{4}
  \node[draw, circle] at ({360/\n * (\s)}:\radius) {\e}; 
 \draw[-, >=latex]   ({360/\n * (\s-1)+\margin}:\radius) 
    arc ({360/\n * (\s-1)+\margin}:{360/\n * (\s)-\margin}:\radius);
\def \s{3} 
\def \e{2}
  \node[draw, circle] at ({360/\n * (\s)}:\radius) {\e}; 
 \draw[-, >=latex]   ({360/\n * (\s-1)+\margin}:\radius) 
    arc ({360/\n * (\s-1)+\margin}:{360/\n * (\s)-\margin}:\radius);
\def \s{4} 
\def \e{3}
  \node[draw, circle] at ({360/\n * (\s)}:\radius) {\e}; 
 \draw[-, >=latex]   ({360/\n * (\s-1)+\margin}:\radius) 
    arc ({360/\n * (\s-1)+\margin}:{360/\n * (\s)-\margin}:\radius);
\def \s{5} 
\def \e{2}
  \node[draw, circle] at ({360/\n * (\s)}:\radius) {\e}; 
 \draw[-, >=latex]   ({360/\n * (\s-1)+\margin}:\radius) 
    arc ({360/\n * (\s-1)+\margin}:{360/\n * (\s)-\margin}:\radius);
\def \s{6} 
\def \e{4}
  \node[draw, circle] at ({360/\n * (\s)}:\radius) {\e}; 
  \draw[-, >=latex]   ({360/\n * (\s-1)+\margin}:\radius) 
    arc ({360/\n * (\s-1)+\margin}:{360/\n * (\s)-\margin}:\radius);
\def \s{7} 
\def \e{3}
  \node[draw, circle] at ({360/\n * (\s)}:\radius) {\e}; 
  \draw[-, >=latex]   ({360/\n * (\s-1)+\margin}:\radius) 
    arc ({360/\n * (\s-1)+\margin}:{360/\n * (\s)-\margin}:\radius);
\end{tikzpicture}

\end{multicols}

\caption{Examining the  Skolem circle of Figure \ref{fig:6sequences}, there are three different edges that could be removed whilst maintaining a Skolem Labeling.   Hence this Skolem circle contains 6 distinct Skolem sequences. \label{fig:edgesremoved}}
\end{figure}
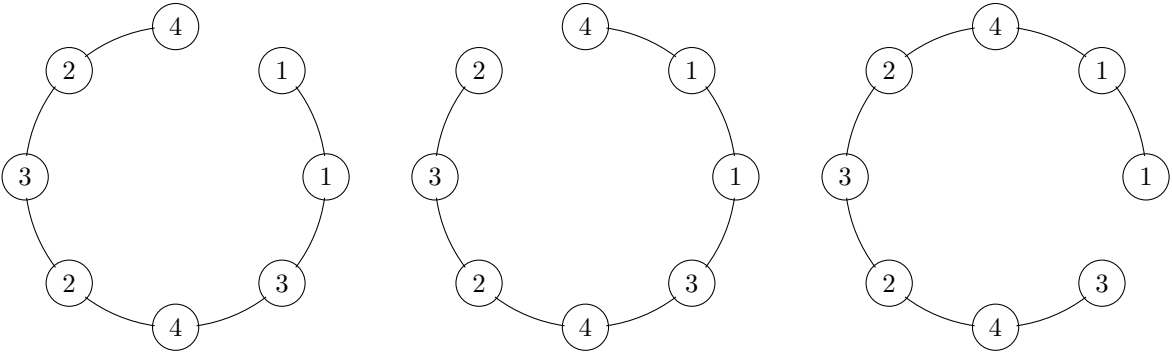

Figure \label{fig:edgesremoved} shows a Skolem circle with three different edges removed, thus creating 6 different Skolem sequences.

Let $C$ be a Skolem circle of order $m$ to which $j$ different edges may be removed to create 2j different Skolem sequences, then $C$ is a $j$-{\em edge- removable} Skolem circle.  A circle with 3 removable edges is shown in Figure \ref{fig:edgesremoved}.  When removing all edges at once the resulting graph may be a Skolem labeled graph, or as in case of Figure \ref{fig:edgesremoved}, may from a extended Skolem Labeled graph by treating the symbol $m$ as a null.

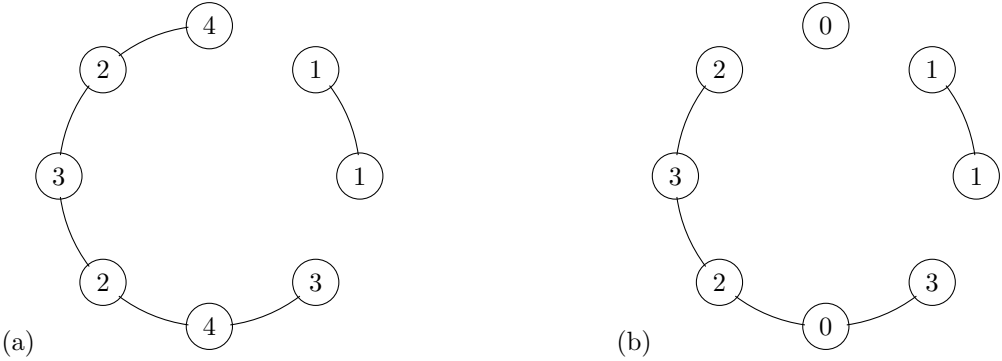
\begin{figure}[t]

\begin{multicols}{2}

(a)\begin{tikzpicture}
\def \n {8}
\def \radius {2cm}
\def \margin {8} 
\def \s{0} 
\def \e{1}
  \node[draw, circle] at ({360/\n * (\s)}:\radius) {\e}; 
    arc ({360/\n * (\s-1)+\margin}:{360/\n * (\s)-\margin}:\radius);
\def \s{1} 
\def \e{1}
  \node[draw, circle] at ({360/\n * (\s)}:\radius) {\e}; 
 \draw[ , >=latex]   ({360/\n * (\s-1)+\margin}:\radius) 
    arc ({360/\n * (\s-1)+\margin}:{360/\n * (\s)-\margin}:\radius);
\def \s{2} 
\def \e{4}
  \node[draw, circle] at ({360/\n * (\s)}:\radius) {\e}; 
    arc ({360/\n * (\s-1)+\margin}:{360/\n * (\s)-\margin}:\radius);
\def \s{3} 
\def \e{2}
  \node[draw, circle] at ({360/\n * (\s)}:\radius) {\e}; 
 \draw[-, >=latex]   ({360/\n * (\s-1)+\margin}:\radius) 
    arc ({360/\n * (\s-1)+\margin}:{360/\n * (\s)-\margin}:\radius);
\def \s{4} 
\def \e{3}
  \node[draw, circle] at ({360/\n * (\s)}:\radius) {\e}; 
 \draw[-, >=latex]   ({360/\n * (\s-1)+\margin}:\radius) 
    arc ({360/\n * (\s-1)+\margin}:{360/\n * (\s)-\margin}:\radius);
\def \s{5} 
\def \e{2}
  \node[draw, circle] at ({360/\n * (\s)}:\radius) {\e}; 
 \draw[-, >=latex]   ({360/\n * (\s-1)+\margin}:\radius) 
    arc ({360/\n * (\s-1)+\margin}:{360/\n * (\s)-\margin}:\radius);
\def \s{6} 
\def \e{4}
  \node[draw, circle] at ({360/\n * (\s)}:\radius) {\e}; 
  \draw[-, >=latex]   ({360/\n * (\s-1)+\margin}:\radius) 
    arc ({360/\n * (\s-1)+\margin}:{360/\n * (\s)-\margin}:\radius);
\def \s{7} 
\def \e{3}
  \node[draw, circle] at ({360/\n * (\s)}:\radius) {\e}; 
  \draw[-, >=latex]   ({360/\n * (\s-1)+\margin}:\radius) 
    arc ({360/\n * (\s-1)+\margin}:{360/\n * (\s)-\margin}:\radius);
\end{tikzpicture}

(b)\begin{tikzpicture}
\def \n {8}
\def \radius {2cm}
\def \margin {8} 
\def \s{0} 
\def \e{1}
  \node[draw, circle] at ({360/\n * (\s)}:\radius) {\e}; 
    arc ({360/\n * (\s-1)+\margin}:{360/\n * (\s)-\margin}:\radius);
\def \s{1} 
\def \e{1}
  \node[draw, circle] at ({360/\n * (\s)}:\radius) {\e}; 
 \draw[ , >=latex]   ({360/\n * (\s-1)+\margin}:\radius) 
    arc ({360/\n * (\s-1)+\margin}:{360/\n * (\s)-\margin}:\radius);
\def \s{2} 
\def \e{0}
  \node[draw, circle] at ({360/\n * (\s)}:\radius) {\e}; 
    arc ({360/\n * (\s-1)+\margin}:{360/\n * (\s)-\margin}:\radius);
\def \s{3} 
\def \e{2}
  \node[draw, circle] at ({360/\n * (\s)}:\radius) {\e}; 
    arc ({360/\n * (\s-1)+\margin}:{360/\n * (\s)-\margin}:\radius);
\def \s{4} 
\def \e{3}
  \node[draw, circle] at ({360/\n * (\s)}:\radius) {\e}; 
 \draw[-, >=latex]   ({360/\n * (\s-1)+\margin}:\radius) 
    arc ({360/\n * (\s-1)+\margin}:{360/\n * (\s)-\margin}:\radius);
\def \s{5} 
\def \e{2}
  \node[draw, circle] at ({360/\n * (\s)}:\radius) {\e}; 
 \draw[-, >=latex]   ({360/\n * (\s-1)+\margin}:\radius) 
    arc ({360/\n * (\s-1)+\margin}:{360/\n * (\s)-\margin}:\radius);
\def \s{6} 
\def \e{0}
  \node[draw, circle] at ({360/\n * (\s)}:\radius) {\e}; 
  \draw[-, >=latex]   ({360/\n * (\s-1)+\margin}:\radius) 
    arc ({360/\n * (\s-1)+\margin}:{360/\n * (\s)-\margin}:\radius);
\def \s{7} 
\def \e{3}
  \node[draw, circle] at ({360/\n * (\s)}:\radius) {\e}; 
  \draw[-, >=latex]   ({360/\n * (\s-1)+\margin}:\radius) 
    arc ({360/\n * (\s-1)+\margin}:{360/\n * (\s)-\margin}:\radius);
\end{tikzpicture}

\end{multicols}
\caption{(a) Removing two edges from the Skolem circle of Figure \ref{fig:6sequences} creates a Skolem labeled graph which is no connected.  (b) Replacing the symbol $4$ with a null, the Skolem circle is partitioned into 3 sequences of Skolem type. \label{fig:disconnected}  }
\end{figure}

 The edges of the Skolem labeled graph that would removed to partition a Skolem circle into    Skolem-type (or extended Skolem-type) subsequences is called a {\em removable edge}.

  At each removable edge two Skolem sequences are constructed (clockwise and anti-clockwise) see Figure \ref{fig:edgesremoved}. All Skolem sequences create a Skolem circle with at least one break point.  The Skolem circle in Figure \ref{fig:6sequences} is $3$ breakable, the anticlockwise subsequences are $(1,1),(4),(2,3,2,4,3)$. The Skolem circle in Figure \ref{fig:unbreakable} has $0$ removable edges, it is Skolem labelling of a graph.

Note that removing an edge from a Skolem circle creates two Skolem sequences, clockwise and anti clockwise. These observations lead to the following result.
\begin{lem}A  Skolem circle with $j$ removable edges represents a set of $2j$ circle equivalent Skolem sequences.
\end{lem}

\begin{ex}\label{ex:4}
Construct a $4$-edge-removeable  Skolem circle  by pasting together Langford Sequences.  Begin with the smallest Langford sequence $(1,1)$.  Next  a Langford sequence of defect $2$, for example $(3,4,2,3,2,4)$.   Then a Langford sequence of defect $5$, for example\\ $(13,11,9,7,5,12,10,8,6,5,7,9,11,13,6,8,10,12)$.\\
 Finally a Langford Sequence of defect $14$ for example\\ $(40,38,36,34,32,30,28,26,24,22,20,18,16,14,39,37,35,33,31,29,27,25,23,21,19,17,15,14,16,$\\ $18,20,22,24,26,28,30,32,34,36,38,40,15,17,19,21,23,25,27,29,31,33,35,37,39)$. 
These sequences can be concatenated together and wrapped around a cycle to form a $4$-edge-removable Skolem circle of order $40$.
\end{ex}

This example can be formalised into a general construction.

\begin{lem}\label{lem:gluing}
Let $\mathcal{S}=\{S_1,S_2,\dots,S_j\}$ be a set of $j$ Skolem-type sequences and let $X_i$ be the set of symbols in sequence $S_i$ for $i\in\{1,2,\dots, j\}$.  If the set 
\[\bigcup_iX_i=\{1,2,3,\dots m\}\quad\quad\textrm{and}\quad\quad \bigcap_iX_i=\emptyset\]
Then the concatenation of the sequences in $S$, wrapped around a cycle forms a $j$-edge-removable Skolem circle. 
\end{lem}

Next we find appropriate Skolem type sequences.

\begin{thm}\label{thm:langford} \cite{simpson1983}
A Langford Sequence of order $m$ and defect $d$ exists if and only if 
\begin{enumerate}
\item $m\geq 2d-1$
\item If $d$ is odd then $m\equiv 0,1\mod 4$, and if $d$ is even then $m\equiv 0,3\mod 4$.
\end{enumerate}
\end{thm}
Knowing that appropriate Langford sequences exist a general construction can be described.

\begin{thm}\label{thm:9}
Let $j\in\mathbb{N}$, then there exists a  $j$-edge-removable Skolem circle of order $((3^{j+1})/2)-1$.
\end{thm}
\begin{proof}
 For $n \geq 0$ let $d_n=(3^n+1)/2$ and $m_n=3^n$.  With these parameters if $n$ is odd then $d_n$ is even and $m_n\equiv 3\mod 4$, and if $n$ is even then $d_n$ is odd and $m_n\equiv 1 \mod 4$.  Thus the conditions of Theorem \ref{thm:langford} are satisfied and therefore Langford sequences with these parameters exist.
 
 Let $S_n$ be a Langford sequence of defect $d_n$ and order $m_n$. Let $X_n$ be the set of symbols contained in $S_n$. Note that $X_0=\{1\}, X_1=\{2,3,4\}$. 
Hence
\[\bigcup_{i=0}^{1}X_i=\{1,2,3,4\}\quad\quad\textrm{and}\quad\quad \bigcap_{i=0}^1 X_i=\emptyset\]
Hence by Lemma \ref{lem:gluing} the concatenation of $S_0, S_1,$ forms a $2-$edge-removable Skolem circle.
Now assume for induction that the concatenation of $S_0,\dots ,S_{n-1}$ is a $n$-edge-removable Skolem circle with order $d_{n-1}+m_{n-1}-1$.
\[d_{n-1}+m_{n-1}-1=(3^{n-1}+1)/2+3^{n-1}-1= (3^{n-1}+1+2\times3^{n-1})/2-1=d_{n}-1\]
Hence 
\[\left(\bigcup_{i=0}^{n-1}X_i\right)\cup X_n=\{1,2,3,4\dots,d_{n}+m_n-1\}\quad\quad\textrm{and}\quad\quad\left( \bigcap_{i=0}^{n-1} X_i\right)\cap X_{n}=\emptyset\]

So by Lemma \ref{lem:gluing} the concatenation of  $S_0,\dots ,S_{n-1},S_{n}$ is a $(n+1)$-edge-removable Skolem circle.  The order of the Skolem circle is the largest value in $S_n$, and is hence $d_{n}+m_n-1=(3^{j+1})/2-1$.   By the principle of induction  the concatenation of  $S_0,\dots ,S_j$ is a $j+1$-edge-removable Skolem circle of order $(3^{j+1})/2-1$.
\end{proof}

We now give an algorithm for a specific construction of the required Langford sequences.
\begin{alg}\label{alg:construction}
To construct a Langford sequence of order $m=3^j$ and defect $d=(3^j+1)/2$.  Let $S_j$ be a sequence of order $m$, with $S_j[i]$ representing the $i^{th}$ position in the sequence.
\begin{itemize}
\item[] or $i\in[1,d]$ 
set $S_j[i]=d+m+1-2i$.
\item[] For $i\in[d+1,m]$ set  $S_j[i]=2m+d+1-2i$
\item [] For $i\in[m+1,3(m+1)/2$ set $S_j[i]=d+2(m+1-i)$

\end{itemize}

\end{alg}

Example \ref{ex:4} is   a $4$-breakable Skolem circle constructed using Algorithm \ref{alg:construction}.
\begin{proof}
If $x$ is the same parity as $d+m-1$, then for $y\in\{0,1,2,\dots, d-1\}$, let $x=d+m-1-2y$ and the Skolem pairs are $(y+1,m+d-y)$.  If $x$ is the opposite  parity as $d+m-1$ then for $z\in\{1,2,d-1\}$ let  $x=d+m-2z$ and the Skolem pairs are $(d+z, 2m-z+1)$.  

The gap between the Skolem pairs is  
\[m+d-y-(y+1)=d+m-1-2y=x\]
when $x$ is the same parity as $d+m-1$ and 
\[2m-z+1-(d-z)= 2m-d+1-2z=m+d-2z=x\]
 when the $x$ is the opposite parity.
\end{proof}

This shows that it is possible to construct a $j$-edge-removeable Skolem circle for any $j\in\mathbb{N}$, the Skolem circle just needs to be large enough.  The next result gives a bound for how large the circle must be.

\begin{thm}\label{thm:breakpointbound}
Let $S$ be a  Skolem circle of order $m$, then the maximum number of removable edges  is $O(\log m)$.
\end{thm}
\begin{proof}
Let $C$ be a $j$-edge-removable Skolem circle with $j\geq 1$.  Let $S_1,S_2,\dots, S_j$ be the $j$ subsequences of Skolem type or extended Skolem type that when concatenated together give $C$.  Let the symbols $\{1,2,\dots ,m\}$ as used in the Skolem circle be treated as integers with the usual ordering.  Let the sequences be ordered by their largest symbol.  Note that for all symbols $x\neq m$, both copies of $x$ are in the same sequence, thus the only ambiguity in the ordering is if the symbol $m$ is the largest symbol in two different subsequences, in which case these two may be ordered according to the next largest symbol (as noted below at most one of the sequences can contain only the symbol $m$ so there is no ambiguity to this ordering).  Note that if a subsequence $S_i$ contains only one of the symbol $m$  then $S_i$ is an extended Skolem sequence with $m$ as the null symbol $m$ is the null symbol, but the symbol will not be changed to $0$ (as in common in extended sequences).  At most two sequences contain $m$, and they will be ordered as $S_j$ (and $S_{j-1}$ if needed), thus $S_{i}$ for $i\leq j-2$ do not contain the symbol $m$, and must be Skolem type sequences.

The largest symbol appearing twice in $S_1$ is at least $1$. Let $x\in\{2,3,\dots,j-3\}$.  We proceed by induction on $x$.

Assume that for every $i\in\{1,2,3,\dots, x-1\}$ the largest symbol appearing in $S_i$ is at least $2^{i-1}$.  Let $t_i$ be the largest symbol appearing in $S_i$, then $t_i\geq 2^{i-1}$ .  Let $X_i$ be the set of symbols which appear in  the sequence $S_i$.  Let $\mathcal{X}=\bigcup _{i=1}^{x}X_i$, then $t_x$ is the largest symbol in $\mathcal{X}$.  Thus $|\mathcal{X}|\leq t_x$.  However the length of a Skolem type sequence  is at least one more than the largest symbol.  Therefore the number of distinct symbols appearing in $S_i$ is at least $(t_i+1)/2$.  Hence
\begin{align}
\sum_{i=1}^x\frac{t_i+1}{2} & \leq   |\mathcal{X}|\leq t_x\\
\left(\sum_{i=1}^{x-1}(t_i+1) \right)+t_x+1 & \leq 2t_x\\
\left(\sum_{i=1}^{x-1}(t_i+1) \right)+1 & \leq t_x
\end{align}

Using the inductive hypothesis
\begin{align}
\left(\sum_{i=1}^{x-1}(2^{i-1}+1) \right)+1 & \leq t_x\\
2^x+x & \leq t_x\\
2^{x-1} & \leq t_x
\end{align}
By induction the largest symbol appearing in $S_{j-2}$ is at least $2^{j-3}$.  

The sequences $S_j$ and $S_{j-1}$ must be non-empty.  Exactly one of the Skolem pair $(a_m,b_m)$ must lie between the Skolem pair $(a_{m-1},b_{m-1})$,  hence at least one of the sequences $S_j$ or $S_{j-1}$ must contain at least $m/2$ different symbols.  Hence $|X_j\cup X_{j-1}|\geq  m/2$.

Hence 
\[m/2\geq |\{1,2,\dots ,m\}\setminus (X_j\cup X_{j-1})|= |\bigcup_{i=1}^{j-2}S_i|\geq \sum_{i=1}^{j-2}(2^{i-1}+1)/2\geq 2^{j-3}\]
   which can be rearranged to $j\leq 2+\log_2(m)$.

\end{proof}

\section{Enumeration of Skolem Circles  \label{sec:enumerate}}
One of the fundamental questions asked about any combinatorial structure is how many?  We use the ideas of circle equivalence and a known bound on the number of Skolem sequences to give a general bound on the number of distinct Skolem circles.  We also compute the number of distinct Skolem circles for small orders.

There is a general lower bound on the number of distinct Skolem sequences of order $m$.
\begin{thm} \cite{abrham1986}\label{thm:abrahambound}
There are at least  $2^{\lfloor\nicefrac{m}{3}\rfloor}$  distinct Skolem sequences of order $m$ 
\end{thm} 
Combine this with Theorem \ref{thm:breakpointbound}, and we have a lower bound on the number of distinct Skolem circles.

\begin{thm}
For $m\geq 8$ with $m\equiv 0,1\mod 4$, there is at least 
\[\frac{2^{\lfloor\nicefrac{m}{3}\rfloor-1}}{2+\log_2(m)}+1\]
Skolem circles of order $m$.
\end{thm}
\begin{proof}
From Theorem \ref{thm:abrahambound} there are at least $2^{\lfloor\nicefrac{m}{3}\rfloor}$  distinct Skolem sequences of order $m$.  From Theorem \ref{thm:breakpointbound} there are at most $2+\log_2(m+1)$ breakpoints in each Skolem circle, and hence each Skolem circle of order $m$ represents at most $2(2+\log_2(m+1))$ Skolem circles.  From \cite{mendelsohn1991} at least one $0$-edge-removable Skolem circle exists.
\end{proof}

Next we enumerate the number of distinct Skolem circles for small orders.  As with many combinatorial problems, the small orders $m=4$ and $m=5$ can be done by hand, then larger orders require computer based techniques.

\begin{thm} There is exactly one Skolem circle of order 4.
\end{thm}
\begin{proof}
There are exactly 6 distinct Skolem sequences of order $4$ \cite{shalaby2007}. The Skolem circle of Figure \ref{fig:6sequences} contains all of these. Hence this is the only edge-removable Skolem circle.   From \cite{mendelsohn1991} we know that there are no $0$-edge-removable Skolem circles of order $4$.
\end{proof}

\begin{thm} There are exactly two Skolem circles of order 5.
\end{thm}
\begin{proof}
There are exactly 10 distinct Skolem sequences of order 5 \cite{shalaby2007}. They can be partitioned into two circle equivalence classes.

The class $\{(1, 1, 5, 2, 4, 2, 3, 5, 4, 3)$, $(5, 2, 4, 2, 3, 5, 4, 3, 1, 1)$,
$(3, 4, 5, 3, 2, 4, 2, 5, 1, 1)$,\\ $(1, 1, 3, 4, 5, 3, 2, 4, 2, 5)\}$ is a $2$-edge-removable Skolem circle.  The class   $\{(1, 1, 5, 4, 2, 3, 2, 5, 3, 4)$,\\ $(5, 4, 2, 3, 2, 5, 3, 4, 1, 1)$, $(4, 1, 1, 5, 4, 2, 3, 2, 5, 3)$, $(4, 3, 5, 2, 3, 2, 4, 5, 1, 1)$, $(3, 5, 2, 3, 2, 4, 5, 1, 1, 4)$,\\ $(1, 1, 4, 3, 5, 2, 3, 2, 4, 5)\}$  is a $2$-edge-removable Skolem circle.

From Theorem \cite{mendelsohn1991} we know that there are no $0$-edge-removable Skolem circles of order $5$.
\end{proof}

For larger values of $m$   computer based searches were used. Computations were done on the QUT High Performance Computing facilities  using Matlab \cite{MATLAB} and C++\cite{C++}.  

Skolem circles (and Skolem sequences) of order $m$ can be classified according to the number of removable edges.  

The following technique was used to search for Skolem circles.  Begin by setting up sets of  logical vectors $A_i$ containing all the possible places for each Skolem pair $(a_i,b_i)$.  Note that the Skolem circle is in standard positional labeling, therefore there is exactly one place for placing the pair of $1$s, and $(m-2)$ for placing a pair of 2s. Thus $|A_1|=1$ and $|A_2|=m-2$, all other sets  $|A_i|=2m-4$, since the symbol 1 is already fixed in positions $1$ and $2$.  Vectors from the set $X_2,X_3,\dots $ are progressively chosen, and checked for any clashes.  In this way a complete list of Skolem circles can be created for small values of $m$. This algorithm was inspired by an algorithm previously used to enumerate Langford Sequences \cite{miller1997}.  Lists of Skolem circles of small order are available on  \url{http://www.joannelhall.com/gallery/skolem}. More computational time would obviously increase the size of the Skolem circles that can be catalogued.

The following algorithm was used to calculate the number of removable-edges in a Skolem circle.

\begin{alg}
\begin{enumerate}
\item Let $C$ be a Skolem circle, with standard labeling and let $X=\{(a_1,b_1),(a_2,b_2),$\\$\dots,(a_m,b_m)\}$ be the ordered pair representation of $C$.  As such $(a_1,b_1)=(1,1)$.

\item Let $\vec{v}$ be a vector of length $2m$ with $v_i$ being the $i^{th}$ component of $\vec{v}$.  For each $i\in\{2,3,4,\dots, m-1\}$  set the components $v_{a_{i}+1}, v_{a_{i}+2},\dots, v_{a_{i}+i-1}$ to $0$, where the subscripts are calculated $\mod 2m$. This sets the components of $\vec{v}$ with positions between Skolem pairs to $0$.   All components not set to $0$ are then set to $1$.  

\item If $\vec{v}=\vec{0}$, then $C$ is an $0$-edge removeable Skolem Circle.  The only components set to $1$ are those which are not between any Skolem pair, and are hence are nodes of a removable edge.

\item If the symbol $m$ has a  removable edge on either side, then the position $a_m$ (or $b_m$) is next to 2 removable edges.  Thus the two removable edges around the null sequence $(m)$ will induce $3$ 1s in the vector $\vec{v}$. Exactly one of $a_m$ or $b_m$ must be between $a_{m-1}$ and $b_{m-1}$, hence two adjacent removable edges can occur at most once in any circle.  All other removable edges are nonadjacent, and hence all other removable edges create a unique pair of adjacent 1s in the vector $\vec{v}$.

\item Let $\textrm{w}(\vec{v})$ be the Hamming weight of $\vec{v}$.  Then $C$ is a $j$-edgeremoveable Skolem circle with $j=\nicefrac{1}{2}\textrm{w}(\vec{v})$ if $\textrm{w}(\vec{v})$ is even,and $j= \nicefrac{1}{2}(\textrm{w}(\vec{v})+1)$ if $\textrm{w}(\vec{v})$ is odd.
\end{enumerate}
\end{alg}

Tables \ref{table:breakpoints} and \ref{table:count} contain the data gathered from this computation.

\begin{table}
\begin{tabular}{crrrrrrr}
removable edges & 0 & 1 & 2 & 3 &4 \\
\hline
m=4 & 0 & 0 &0 &  1 & 0\\
m= 5 & 0  & 0 &  1 & 1 &  0\\
m=8 & 24 & 96 & 60 & 12 &0 \\
m=9 & 280 &  574 & 284 & 62 & 0\\
m=12 &271,880 & 146,436 & 34,400 & 4,244 &0 \\
m= 13 & 2,742,984 & 1,035,186 & 207,756 & 22,810 & 288\\
m=16 & 3,764,810,632 & 530,928,868 & 75,697,744 & 5,872,996 & 33,760\\
m=17 & 46,071,353,270 & 4,751,383,672 & 620,552,462 & 43,754,420 & 184,848
\end{tabular}

\caption{\label{table:breakpoints}  The number of distinct Skolem circles of each order, classified according to the number of breakpoints}
\end{table}

\begin{table}
\begin{tabular}{crrr}
 & Skolem sequences \cite{shalaby2007} &  Skolem circles\\
\hline
m=4 & 6 & 1 \\
m=5 & 10 & 2 \\
m=8 & 504 & 192\\
m=9 & 2,656 & 1,200 \\
m=12 & 455,936 & 456,960 \\
m=13 & 3,040,560 & 4,009,024\\
m=16 & 1,400,156,768 & 4,377,344,000\\
m=17 & 12,248,982,496 & 51,487,228,672
\end{tabular}
\caption{\label{table:count} The number of distinct Skolem sequences and Skolem circles}
\end{table}

From Table \ref{table:count} note that for $m\geq 12$ the number of Skolem circles is greater than the number of Skolem sequences. 

A simple combinatorial argument shows that the number of  Skolem circles grows faster than the number of Skolem sequences.

Begin with an empty sequence with 2m positions to fill.  Then $a_1$ may take any value other than $2m$, this is $(2m-1)$ different possible values for $a_1$,  with $b_1=a_1+1$.    The symbol $a_2$ can take any value other that $2m, 2m-1$, thus there is a maximum of $2m-2$ different possible values for $a_2$, then $b_2=a_2+2$.  The symbol $a_3$ can take any value other that $2m, 2m-1,2m-2$, thus there is a maximum of $2m-3$ different possible values for $a_3$. The symbol $a_i$ with $i\geq 3$ can take any value other than $2m,2m-1,\dots, 2m-i+1$,  there is a maximum of $2m-i$ possible values for $a_1$. For $i=m-2$ there is a maximum of 2 possible arrangement for the symbols $m,m-1,m-2$\cite{miller1997}.

 Taking the product of these possibilities we see that there is a maximum of 
\[2\prod_{i=1}^{m-2}\left(2m-i\right)\sim O(m!)\]
possible Skolem sequences.

We use the same counting techniques to find an upper bound on the number of Skolem circles.  A Skolem circle always begins with $a_1=1$ and $b_1=2$.   Due to standard labelling, the symbol $a_2$ can take values from $\{3,4,\dots ,m\}$ thus there is a maximum of $m-2$ different possible values for $a_2$,  then  $b_2\equiv a_2+2\mod 2m$
 The symbol $a_i$ can take any value other that  $1$ and $2$, thus there is a maximum of $2m-2$ different possible values for $a_i$,  then  $b_i\equiv a_i+i\mod 2m$.  For $i=m-2$ there is a maximum of 2 possible arrangements for the symbols $m,m-1,m-2$\cite{miller1997}. Taking the product of these  possibilities we see that there is a maximum of 
\[2(m-2)\left(2m-2\right)^{m-5}\sim O(m^m) \]
possible Skolem circles.

These are both naive upper bounds, however they are calculated in the same way, so give us some intuition. 
\[O(m^m)>O(m!)\]
And hence the number of Skolem circles eventually grows faster than the number of Skolem sequences.  

\section{Further Ideas}
The motivation behind investigating Skolem circles was to portray sequences with symbols taken from a finite cyclic group, which was found to correlate well with previous work on Skolem labeling of circle graphs.  Investigating Skolem structures  with symbols taken from other group structures may correlate with Skolem labeling of other families of graph.

A similar investigation of Langford or Rosa  labeling of cycle graphs may also lead in interesting directions.

\section*{Acknowledgement}
J. Bubear was supported by an AMSI Vacation Research Scholarship.
Computational  resources  used in this work were provided by the High Performance Computing and Research Support Group, Queensland University of Technology, Brisbane, Australia.
Thanks to Diane Donavan and Asha Rao for discussions which lead to the idea of a Skolem circle.  Thanks to Daryn Bryant for suggesting the construction of pasting Langford sequences together.   Thanks to  Keving Hendrey and  Tim Wilson for describing  the proof of Theorem \ref{thm:9}. Thanks to the organisers of the 39th Australasian Conference on Combinatorial Mathematics and Combinatorial Computing, where an earlier version of this research was presented.

\end{document}